\documentclass[11pt]{article}
\usepackage[margin=1in]{geometry}
\usepackage[T1]{fontenc}
\usepackage{lmodern}
\usepackage{amsmath,amssymb,amsthm}
\usepackage{graphicx,tikz,needspace}
\usetikzlibrary{arrows.meta,positioning,calc}
\definecolor{XYInk}{HTML}{243746}
\definecolor{XYBlue}{HTML}{246A8D}
\definecolor{XYRust}{HTML}{B65A35}
\definecolor{XYMuted}{HTML}{CBD3D8}
\definecolor{XYGreen}{HTML}{4A785E}
\usepackage[colorlinks=true,linkcolor=blue,citecolor=blue,urlcolor=blue]{hyperref}

\newtheorem{theorem}{Theorem}[section]
\newtheorem{lemma}[theorem]{Lemma}
\newtheorem{proposition}[theorem]{Proposition}
\newtheorem{corollary}[theorem]{Corollary}
\theoremstyle{remark}
\newtheorem{remark}[theorem]{Remark}
\newcommand{\E}{\mathbb E}
\newcommand{\Z}{\mathbb Z}
\newcommand{\T}{\mathbb T}
\newcommand{\Cov}{\operatorname{Cov}}
\newcommand{\bnd}{\mathrm{bond}}
\newcommand{\sit}{\mathrm{site}}
\newcommand{\eff}{\mathrm{eff}}
\newcommand{\av}[1]{\langle #1\rangle}

\title{Correlation comparisons and critical curves\\
for disordered XY models}
\author{Yan Ru Pei\\\small\href{mailto:yanrpei@gmail.com}{\texttt{yanrpei@gmail.com}}}
\date{}
\hypersetup{
  pdftitle={Correlation comparisons and critical curves for disordered XY models},
  pdfauthor={Yan Ru Pei}
}

\begin{document}
\maketitle

\begin{abstract}
We prove a lower comparison for disorder-averaged ferromagnetic XY
correlations: independent nonnegative random couplings $K_e$ can be replaced
by $-\frac12\log\E e^{-2K_e}$. Ginibre's inequality makes each normalized
correlation convex in $1-e^{-2K_e}$, so conditional Jensen and thinning
compare Bernoulli bond densities; two star-wise Jensen passes handle site
dilution on bipartite graphs. Together with the phase-existence theorem of
Dario and Garban, these comparisons give local Lipschitz regularity of
both the averaged algebraic-phase and susceptibility-divergence thresholds
throughout the supercritical density intervals on $\Z^2$.
The clean finite-domain criterion of van Engelenburg and Lis also yields
explicit finite-size susceptibility lower bounds in a weak-dilution region.
For bounded centered bond noise, the possible suppression of the
algebraic-phase temperature is at most quadratic in the noise amplitude.
Finally, at finite positive thresholds we prove total-variation continuity
for bounded iid coupling laws whose essential supremum equals the common support cap,
with local Lipschitz control when the reference law has an atom at that cap.
\end{abstract}

\section{Introduction and main result}

The two-dimensional ferromagnetic XY model has an algebraic phase at low
temperature. A natural question is whether this phase survives a small
density of missing bonds or sites, even when the temperature is close to
the clean transition. More generally, one may ask how its critical
temperature varies with the density of the underlying percolation.
The issue concerns correlations at arbitrarily large distances, so
continuity of a Gibbs expectation on each fixed finite graph does not by
itself answer either question.

Dario and Garban \cite{DG25} proved that the averaged XY two-point function
has algebraic decay at sufficiently low temperature on every
supercritical Bernoulli site or bond percolation on $\Z^2$.
Their Theorem~1.3 thus gives a finite transition threshold at each
supercritical density. Their Open Question~4 asks whether weak dilution
preserves the algebraic phase at every inverse temperature strictly above
the clean critical point. Dario, van Engelenburg and Garban
\cite[Section~1.4, Open Questions~2--3]{DvEG26} ask this question again and
ask whether the averaged critical curve for bond dilution is continuous.

This note addresses the XY case of these questions by proving explicit
comparisons between different densities. The temperature loss in the
comparison tends to zero with the density loss, uniformly in the volume
and in the locations of the two spins. Besides weak-dilution stability,
the comparison gives local Lipschitz regularity of the critical curve,
for both bond and site dilution. The phase-existence input throughout the
supercritical interval is the theorem of Dario and Garban; the additional
ingredient is a transformed convexity consequence of Ginibre's correlation
inequality.

The same comparison yields two further consequences. First, the clean
finite-domain criterion of van Engelenburg and Lis \cite{vEL23} transfers
to explicit lower bounds on magnetic susceptibility in finite boxes,
and the Laplace-transform formula bounds the effect of weak centered
bond noise. Second, thinning bounded strengths and comparing their
scalar tails gives continuity of finite positive critical thresholds under
changes in the whole coupling law. Continuity holds when the reference
law has essential supremum equal to the common support cap, with a
Lipschitz bound when it has an atom there. These applications are proved
in Sections~\ref{sec:physical} and~\ref{sec:law}.

Differentiating one coupling and then averaging independent couplings has
classical Ising precedents. Alexander, Cesi, Chayes, Maes and Martinelli
\cite[Proposition~1.4 and Equation~(1.16)]{ACCMM98} use the relation
$p=1-e^{-2J}$ to prove concavity in the raw coupling and an upper
comparison for wired FK observables; see also \cite[Lemma~9.6]{GHM01}.
The exponential parameter and Jensen mechanism are standard.
Our use of them for XY correlations controls the additional second-harmonic
covariance by Ginibre's inequality, producing the lower comparison and
its consequences for dilution.

We first specify the critical curve. For a finite nearest-neighbor
subgraph $G=(V,E)$ of $\Z^2$ and couplings $J_e\ge0$, let
\begin{equation}\label{eq:gibbs}
 d\mu_{G,J}(\theta)
 =\frac1{Z_{G,J}}
   \exp\left(\sum_{e=\{u,v\}\in E}J_e\cos(\theta_u-\theta_v)\right)
   \prod_{v\in V}\frac{d\theta_v}{2\pi},
 \qquad \theta\in\T^V,
\end{equation}
where $\T=\mathbb R/(2\pi\mathbb Z)$. We use free boundary conditions.
In the bond model, $J_e=\beta\omega_e$ with independent
Bernoulli($p$) variables $\omega_e$. In the site model,
$J_{\{u,v\}}=\beta\eta_u\eta_v$ with independent Bernoulli($p$)
variables $\eta_v$. All vertices retain their angle variables, including
vacant sites, which are isolated and carry independent Haar angles.
Thus a two-point correlation between distinct vertices with a vacant
endpoint is zero.

For $\star\in\{\bnd,\sit\}$, write
\begin{equation}\label{eq:F}
 F_\star(p,\beta;x,y)
 =\E_p\av{\cos(\theta_x-\theta_y)}_{\beta,\star}
\end{equation}
for the free infinite-volume limit of the averaged correlation.
Each Gibbs expectation is normalized before averaging over the disorder.
Existence of this limit follows from Ginibre monotonicity; the details
needed for the comparisons are given in Section~\ref{sec:dilution}.
Define
\begin{equation}\label{eq:critical}
 \beta_{c,\star}(p)
 =\inf\left\{\beta\ge0:
 \begin{array}{l}
 \text{there exist }c>0\text{ and }s\in[0,\infty)\text{ such that}\\[-2pt]
 F_\star(p,\beta;0,x)\ge c(1+|x|)^{-s}
 \quad\text{for every }x\in\Z^2\setminus\{0\}
 \end{array}\right\},
\end{equation}
with $\inf\varnothing=\infty$.
At each finite positive $\beta$, a polynomial upper bound follows by
comparison with the clean model and the McBryan--Spencer bound, as used
in \cite[Equation~(4.9)]{DG25}. Hence the lower bound in
\eqref{eq:critical} specifies the averaged algebraic phase in
\cite[Equation~(1.18)]{DvEG26}. At $p=1$, the clean
exponential/algebraic dichotomy gives
$\beta_{c,\bnd}(1)=\beta_{c,\sit}(1)=\beta_{\mathrm{BKT}}$;
see \cite[Section~1.2]{DG25}.

For $\kappa>0$, set
\begin{equation}\label{eq:H}
 H_\kappa(t)=1-e^{-2\kappa t},\qquad
 B_\kappa(r,t)=H_\kappa^{-1}\bigl(rH_\kappa(t)\bigr)
 =-\frac1{2\kappa}\log\bigl(1-r(1-e^{-2\kappa t})\bigr),
\end{equation}
where $t\ge0$ and $r\in[0,1]$.

\Needspace{14\baselineskip}
\begin{theorem}[Regularity of the averaged critical curve]\label{thm:critical}
For bond dilution on $\Z^2$, put $(\alpha,\kappa)=(1,1)$;
for site dilution, put $(\alpha,\kappa)=(2,4)$.
Let $p_{c,\star}$ be the corresponding Bernoulli percolation threshold.
On $(p_{c,\star},1]$, the function $\beta_{c,\star}$ is finite and
nonincreasing, while
\begin{equation}\label{eq:critical-monotone}
 p\longmapsto p^\alpha H_\kappa\bigl(\beta_{c,\star}(p)\bigr)
\end{equation}
is nondecreasing. In particular, if
$p_{c,\star}<a\le p\le q\le b\le1$ and
$M=\beta_{c,\star}(a)$, then
\begin{equation}\label{eq:lipschitz}
 0\le\beta_{c,\star}(p)-\beta_{c,\star}(q)
 \le \frac{\alpha e^{2\kappa M}}{2\kappa a}(q-p).
\end{equation}
Consequently $\beta_{c,\star}$ is locally Lipschitz on the
supercritical interval and continuous from the left at $p=1$.
\end{theorem}

\begin{corollary}[Weak-dilution stability]\label{cor:weak}
Use the parameters $(\alpha,\kappa)$ from Theorem~\ref{thm:critical}.
If $p\in(0,1]$ and $\beta\ge0$ satisfy
\begin{equation}\label{eq:weak}
 p^\alpha H_\kappa(\beta)
 \ge H_\kappa(\beta_{\mathrm{BKT}}),
\end{equation}
then the free infinite-volume averaged correlation obeys
\begin{equation}\label{eq:explicit-lower}
 F_\star(p,\beta;x,y)\ge\frac1{8\|x-y\|_\infty}
 \qquad(x\ne y).
\end{equation}
In particular, the algebraic phase survives all sufficiently weak dilution
at every $\beta>\beta_{\mathrm{BKT}}$.
\end{corollary}

The equality case in \eqref{eq:weak} follows from the clean endpoint
theorem of \cite[Theorem~1(ii) and Section~6]{vEL23}. It concerns the
boundary of this explicit sufficient region, not attainment of the
unknown diluted threshold in \eqref{eq:critical}.

The proof rests on the finite-graph estimate
\begin{equation}\label{eq:density-overview}
 F_{G,\star}(p,\beta;x,y)
 \ge F_{G,\star}\bigl(q,B_\kappa((p/q)^\alpha,\beta);x,y\bigr),
 \qquad 0<p\le q\le1.
\end{equation}
For a single variable coupling $t$, write
$f(t)=\av O_t$ for an integer-charge cosine $O$, and let $b=\cos(\theta_u-\theta_v)$.
Ginibre's inequality gives $f'\ge0$ and
\[
 f''+2f'
 =\Cov_t(O,b^2)+2(1-\av b_t)f'\ge0.
\]
The second harmonic in $b^2$ is covered by the same inequality.
Thus $f$ is convex in $H_1(t)$, and Jensen gives a deterministic lower
comparison after averaging one random coupling. Applying it edge by edge
proves the bond comparison. A star version of this observation, applied
successively to the two color classes, proves the site comparison.
After the free-volume limit, the critical-curve conclusions follow from
the order properties of the phase in \eqref{eq:critical}.

\section{A convexity consequence of Ginibre's inequality}
\label{sec:convexity}

We now allow $G=(V,E)$ to be any finite graph. Orient each edge
arbitrarily and write $b_e=\delta_u-\delta_v\in\Z^V$ for its incidence
vector. For $a\in\Z^V$, set
$O_a(\theta)=\cos(a\cdot\theta)$ and
$C_a(J)=\av{O_a}_{G,J}$.
The form of Ginibre's inequality used below is
\begin{equation}\label{eq:ginibre}
 C_a(J)\ge0,\qquad
 \Cov_{G,J}(O_a,O_c)\ge0
 \quad\text{for all }a,c\in\Z^V,\ J\ge0.
\end{equation}
This is the classical correlation inequality of \cite{Gin70}; see also
\cite[Theorem~2.3]{DG25}, whose general cosine interaction formulation
includes arbitrary integer charges. In particular, \eqref{eq:ginibre}
applies to the higher harmonics produced by multiplying edge cosines.

\begin{lemma}[Transformed convexity]\label{lem:convex}
Fix nonnegative couplings $J^0_e$ and numbers $w_e\ge0$ with
$\sum_e w_e\le\kappa$, where $\kappa>0$. Set
\[
 J_e(t)=J^0_e+t w_e,\qquad
 S=\sum_e w_e\cos(b_e\cdot\theta),\qquad
 f(t)=C_a(J(t)),\quad t\ge0.
\]
Then $f'\ge0$, $f''+2\kappa f'\ge0$, and the function
$h\mapsto f(H_\kappa^{-1}(h))$ is convex on $[0,1)$.
\end{lemma}

\begin{proof}
Differentiation under the finite-dimensional Gibbs integral gives
\begin{equation}\label{eq:derivatives}
 f'=\Cov_t(O_a,S),\qquad
 f''=\Cov_t(O_a,S^2)-2\av S_t f'.
\end{equation}
The first derivative is nonnegative by \eqref{eq:ginibre}. Moreover,
\[
 S^2=\frac12\sum_{e,f}w_ew_f
 \left[\cos\bigl((b_e+b_f)\cdot\theta\bigr)
       +\cos\bigl((b_e-b_f)\cdot\theta\bigr)\right],
\]
so $\Cov_t(O_a,S^2)\ge0$ by the same inequality. Since
$\av S_t\le\sum_e w_e\le\kappa$, we obtain
\begin{equation}\label{eq:key}
 f''+2\kappa f'
 =\Cov_t(O_a,S^2)+2(\kappa-\av S_t)f'\ge0.
\end{equation}
For $h=H_\kappa(t)$, the chain rule gives
\[
 \frac{d^2}{dh^2}f(H_\kappa^{-1}(h))
 =\frac{f''(t)+2\kappa f'(t)}{4\kappa^2(1-h)^2}\ge0.
\]
\end{proof}

\begin{remark}[Ising antecedent]
For an Ising spin-product correlation with one bond $t$ varying,
$b=\sigma_u\sigma_v$ satisfies $b^2=1$.
The Griffiths inequalities therefore give
$f''+2f'=2(1-\av b_t)f'\ge0$.
Thus convexity in $1-e^{-2t}$, and the corresponding effective-coupling
lower comparison by Jensen, follow directly in the Ising case.
In the XY calculation, the only extra covariance term is controlled by
the second harmonic in $b^2$. This observation does not assert priority
for an Ising comparison.
\end{remark}

\begin{theorem}[Independent random couplings]\label{thm:random}
Let $(K_e)_{e\in E}$ be independent nonnegative random variables,
each finite almost surely. Define
\begin{equation}\label{eq:effective}
 K_e^\eff=-\frac12\log\E e^{-2K_e}.
\end{equation}
For every $a\in\Z^V$,
\begin{equation}\label{eq:random-comparison}
 \E C_a(K)\ge C_a(K^\eff).
\end{equation}
\end{theorem}

\begin{proof}
Varying just one edge coupling in Lemma~\ref{lem:convex}, with
$w_e=1$ and $\kappa=1$, shows that $C_a$ is separately convex in
$h_e=H_1(J_e)$. Suppose first that the $K_e$ are bounded.
Conditional Jensen replaces $H_1(K_e)$ by its expectation, one edge at a
time. Independence ensures that each replacement leaves the laws of the
remaining coordinates unchanged. This proves
\[
 \E C_a(K)\ge
 C_a\bigl((H_1^{-1}(\E H_1(K_e)))_{e\in E}\bigr),
\]
which is \eqref{eq:random-comparison}.
For general $K_e$, apply the bounded result to $K_e\wedge R$.
As $R\to\infty$, bounded convergence gives convergence of the averaged
correlations, since $0\le C_a\le1$. Also
$\E e^{-2(K_e\wedge R)}\to\E e^{-2K_e}>0$, and continuity in the
finitely many deterministic couplings gives the right-hand limit.
\end{proof}

In particular, an independent Bernoulli($r$) coupling $K_e=t\zeta_e$
has $K_e^\eff=B_1(r,t)$. No moment assumption on $K_e$ is needed in
Theorem~\ref{thm:random}. The independence hypothesis concerns the
couplings themselves; site dilution requires the additional argument below.

\section{Dilution and the free-volume limit}\label{sec:dilution}

For finite $G$, write $F_{G,\star}(p,\beta;a)$ for the expectation of
$C_a$ under bond or site dilution as in \eqref{eq:gibbs}.
For a two-point charge, we also use the endpoint notation
$F_{G,\star}(p,\beta;x,y)$.

\begin{proposition}[Comparison between densities]\label{prop:density}
Let $0<p\le q\le1$ and $\beta\ge0$.
On every finite graph,
\begin{equation}\label{eq:bond}
 F_{G,\bnd}(p,\beta;a)
 \ge F_{G,\bnd}\bigl(q,B_1(p/q,\beta);a\bigr).
\end{equation}
If $G$ is bipartite with maximum degree at most $\Delta\ge1$, then
\begin{equation}\label{eq:site}
 F_{G,\sit}(p,\beta;a)
 \ge F_{G,\sit}\bigl(q,B_\Delta((p/q)^2,\beta);a\bigr).
\end{equation}
These bounds hold for every $a\in\Z^V$.
\end{proposition}

\begin{proof}
For bonds, sample a Bernoulli($q$) graph and independently retain each
present edge with probability $r=p/q$. Conditional on the first graph,
Theorem~\ref{thm:random} replaces each surviving random coupling
$\beta\zeta_e$ by $B_1(r,\beta)$. Averaging over the first graph
gives \eqref{eq:bond}.

For sites, first prove the comparison with the clean graph at retention
$r\in(0,1]$. Let $V=A\sqcup B$ be a bipartition. Conditional on the
site variables $(\eta_v)_{v\in B}$, assign a coupling parameter $t_u$
to each $u\in A$, so the interaction is
\[
 \sum_{u\in A}t_u S_u,\qquad
 S_u=\sum_{v:\{u,v\}\in E}\eta_v\cos(\theta_u-\theta_v).
\]
Every $S_u$ has nonnegative coefficients summing to at most $\Delta$.
Lemma~\ref{lem:convex} and successive Jensen inequalities in the
independent parameters $t_u=\beta\eta_u$ replace all of them by
$t_1=B_\Delta(r,\beta)$. After this replacement the interaction is
\[
 t_1\sum_{v\in B}\eta_v
       \sum_{u:\{u,v\}\in E}\cos(\theta_u-\theta_v).
\]
Apply the same argument to the independent parameters $t_1\eta_v$,
$v\in B$. Each star again has at most $\Delta$ terms. The resulting
deterministic edge coupling is
\[
 t_2=B_\Delta(r,t_1)
 =B_\Delta(r,B_\Delta(r,\beta))
 =B_\Delta(r^2,\beta),
\]
where the last equality follows directly from \eqref{eq:H}.
This proves the clean comparison, also when $\beta=0$.

Now sample independent Bernoulli($q$) site variables $\xi_v$, followed
by independent Bernoulli($p/q$) site variables $\zeta_v$.
Their products $\xi_v\zeta_v$ are independent Bernoulli($p$) variables.
Conditional on $\xi$, remove all edges with an unoccupied endpoint and
retain all vertex integrations. The remaining graph is bipartite with
maximum degree at most $\Delta$. Apply the clean site comparison on
this graph and average over $\xi$. This is \eqref{eq:site}.
\end{proof}

The two-pass mechanism is illustrated in Figure~\ref{fig:site}.

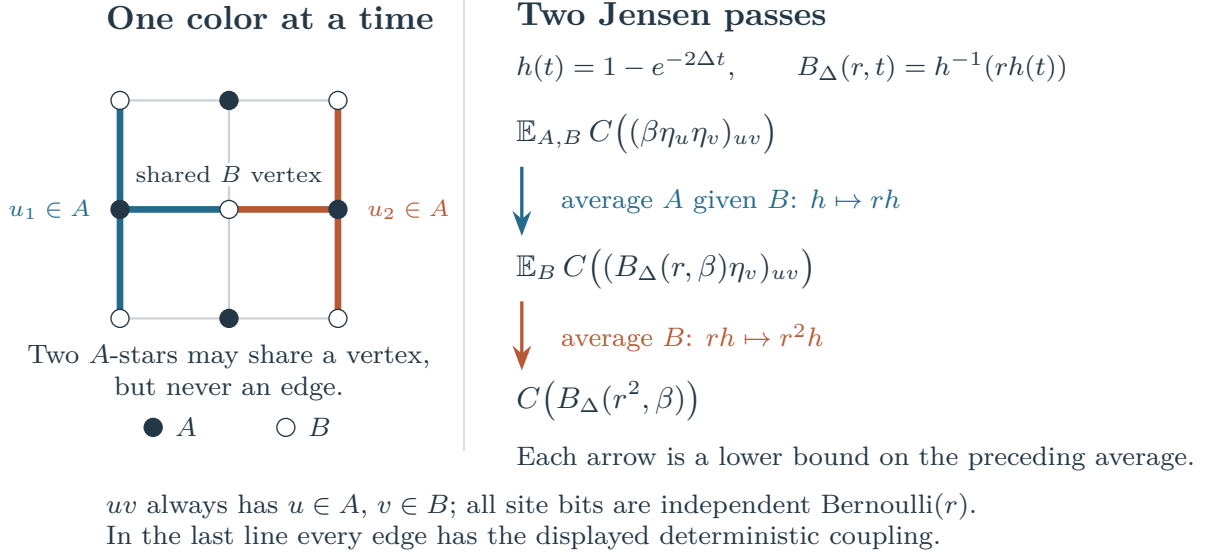
\begin{figure}[tbp]
\centering
\resizebox{.98\linewidth}{!}{%
\begin{tikzpicture}[font=\fontsize{9}{11}\selectfont,text=XYInk,>=Stealth,
  avertex/.style={circle,draw=XYInk,fill=XYInk,inner sep=0pt,minimum size=5.4pt},
  bvertex/.style={circle,draw=XYInk,fill=white,inner sep=0pt,minimum size=5.4pt}]
  \node[anchor=west,font=\fontsize{10}{12}\selectfont\bfseries] at (-.28,3.18) {One color at a time};
  \begin{scope}[x=1.15cm,y=1.15cm]
    \foreach \x in {0,1,2} {
      \foreach \y in {0,1} {\draw[XYMuted,line width=.7pt] (\x,\y)--(\x,{\y+1});}}
    \foreach \y in {0,1,2} {
      \foreach \x in {0,1} {\draw[XYMuted,line width=.7pt] (\x,\y)--({\x+1},\y);}}
    \foreach \p in {(0,0),(0,2),(1,1)} {\draw[XYBlue,line width=2pt] (0,1)--\p;}
    \foreach \p in {(2,0),(2,2),(1,1)} {\draw[XYRust,line width=2pt] (2,1)--\p;}
    \foreach \x in {0,1,2} {\foreach \y in {0,1,2} {
      \pgfmathtruncatemacro{\parity}{mod(\x+\y,2)}
      \ifnum\parity=1 \node[avertex] at (\x,\y) {};
      \else \node[bvertex] at (\x,\y) {};\fi}}
    \node[left=5pt,font=\fontsize{7}{8}\selectfont,text=XYBlue] at (0,1) {$u_1\in A$};
    \node[right=5pt,font=\fontsize{7}{8}\selectfont,text=XYRust] at (2,1) {$u_2\in A$};
    \node[above=7pt,font=\fontsize{7}{8}\selectfont,fill=white,inner sep=1pt] at (1,1) {shared $B$ vertex};
  \end{scope}
  \node[align=center,text width=4.3cm,font=\fontsize{8}{9.5}\selectfont] at (1.15,-.57)
    {Two $A$-stars may share a vertex,\\but never an edge.};
  \node[avertex,label={[font=\fontsize{8}{9.5}\selectfont]right:$A$}] at (.35,-1.15) {};
  \node[bvertex,label={[font=\fontsize{8}{9.5}\selectfont]right:$B$}] at (1.75,-1.15) {};
  \draw[XYMuted] (3.63,-1.4)--(3.63,3.4);

  \node[anchor=west,font=\fontsize{10}{12}\selectfont\bfseries] at (4.05,3.18) {Two Jensen passes};
  \node[anchor=west,font=\fontsize{8}{9.5}\selectfont] at (4.05,2.66)
    {$h(t)=1-e^{-2\Delta t},\qquad B_\Delta(r,t)=h^{-1}(r h(t))$};
  \node[anchor=west] (first) at (4.05,1.95)
    {$\displaystyle \mathbb E_{A,B}\,C\bigl((\beta\eta_u\eta_v)_{uv}\bigr)$};
  \draw[->,XYBlue,line width=1pt] (4.24,1.60)--(4.24,.87);
  \node[anchor=west,font=\fontsize{8}{9.5}\selectfont,text=XYBlue] at (4.52,1.23)
    {average $A$ given $B$: $h\mapsto rh$};
  \node[anchor=west] (second) at (4.05,.53)
    {$\displaystyle \mathbb E_B\,C\bigl((B_\Delta(r,\beta)\eta_v)_{uv}\bigr)$};
  \draw[->,XYRust,line width=1pt] (4.24,.18)--(4.24,-.55);
  \node[anchor=west,font=\fontsize{8}{9.5}\selectfont,text=XYRust] at (4.52,-.20)
    {average $B$: $rh\mapsto r^2h$};
  \node[anchor=west] at (4.05,-.87)
    {$\displaystyle C\bigl(B_\Delta(r^2,\beta)\bigr)$};
  \node[anchor=west,font=\fontsize{8}{9.5}\selectfont] at (4.05,-1.48)
    {Each arrow is a lower bound on the preceding average.};
  \node[anchor=west,font=\fontsize{8}{9.5}\selectfont,text width=11cm,align=left] at (-.28,-2.14)
    {$uv$ always has $u\in A$, $v\in B$; all site bits are independent Bernoulli$(r)$.\\
     In the last line every edge has the displayed deterministic coupling.};
\end{tikzpicture}
}
\caption{Stars in one bipartition class have disjoint edge sets even when they
share a vertex. Sequential Jensen first averages one class and then the other,
multiplying the transformed coupling by $r$ at each pass. The arrows compare
averaged correlations, not individual disorder configurations. Here $C$ is a
fixed integer-charge cosine correlation.}
\label{fig:site}
\end{figure}

To pass to $\Z^2$, couple the disorder in all finite boxes by a single
lattice configuration. Enlarging a box first adds independent Haar
spins and then adds nonnegative edge couplings. The first operation does
not change a correlation supported in the smaller box; the second can
only increase it, by \eqref{eq:ginibre}. Thus each fixed-disorder
correlation has a limit along a free exhaustion, with values in $[0,1]$.
This is also the domain-monotonicity argument in
\cite[Remark~5]{DG25}.
Bounded convergence passes the limit through the disorder expectation
on both sides of Proposition~\ref{prop:density}.
For either model we therefore have, simultaneously for every fixed pair
$x,y\in\Z^2$,
\begin{equation}\label{eq:infinite-comparison}
 F_\star(p,\beta;x,y)
 \ge F_\star\bigl(q,B_\kappa((p/q)^\alpha,\beta);x,y\bigr),
\end{equation}
with the parameters in Theorem~\ref{thm:critical}.
The comparison temperature is independent of both box size and endpoint
separation. Consequently any uniform polynomial lower bound on its
right-hand side holds on its left-hand side with the same constants.

\section{Proof of the critical-curve conclusions}\label{sec:critical-proof}

Fix one of the two dilution models and abbreviate
$c(p)=\beta_{c,\star}(p)$ and $H=H_\kappa$.
Ginibre's inequality and the usual coupling of Bernoulli variables show
that $F_\star$ is nondecreasing in $\beta$ and $p$.
The set of temperatures in \eqref{eq:critical} is therefore upward
closed, and $c(p)$ is nonincreasing. By
\cite[Theorem~1.3, including the bond extension]{DG25}, this set is
nonempty whenever $p>p_{c,\star}$. In particular, $c(p)<\infty$
there, and every $\beta>c(p)$ belongs to the set.
Comparison with the clean model also gives
$c(p)\ge c(1)=\beta_{\mathrm{BKT}}>0$.

We show that $p^\alpha H(c(p))$ is nondecreasing. Fix
$p_{c,\star}<p\le q\le1$. If
$q^\alpha H(c(q))\ge p^\alpha$, then
\[
 p^\alpha H(c(p))\le p^\alpha\le q^\alpha H(c(q)).
\]
Otherwise, choose $\beta_q>c(q)$ decreasing to $c(q)$ such that
$q^\alpha H(\beta_q)<p^\alpha$, and put
\[
 \beta_p=H^{-1}\bigl((q/p)^\alpha H(\beta_q)\bigr).
\]
The temperature on the right of \eqref{eq:infinite-comparison}, with
initial temperature $\beta_p$, is exactly $\beta_q$.
Since $(q,\beta_q)$ has a polynomial lower bound, so does
$(p,\beta_p)$. Hence $c(p)\le\beta_p$, and taking
$\beta_q\downarrow c(q)$ gives
\begin{equation}\label{eq:order}
 p^\alpha H(c(p))\le q^\alpha H(c(q)).
\end{equation}
This use of temperatures strictly above the infimum does not require
the algebraic phase to include its critical point.

Let $p_{c,\star}<a\le p\le q\le b\le1$ and put $M=c(a)$.
The two monotonicities imply
\begin{align*}
 0\le H(c(p))-H(c(q))
 &\le \bigl(1-(p/q)^\alpha\bigr)H(c(p))\\
 &\le 1-(p/q)^\alpha
 \le \frac{\alpha}{a}(q-p).
\end{align*}
Here the last inequality uses $\alpha\in\{1,2\}$ and
$1-r^\alpha\le\alpha(1-r)$ for $r\in[0,1]$.
Since $0\le c(q)\le c(p)\le M$ and
$H'(t)\ge2\kappa e^{-2\kappa M}$ on $[0,M]$, integration gives
\[
 2\kappa e^{-2\kappa M}\bigl(c(p)-c(q)\bigr)
 \le H(c(p))-H(c(q)).
\]
This proves \eqref{eq:lipschitz} and Theorem~\ref{thm:critical}.

Finally, under \eqref{eq:weak},
$B_\kappa(p^\alpha,\beta)\ge\beta_{\mathrm{BKT}}$.
The clean theorem \cite[Theorem~1(ii) and Section~6]{vEL23} gives
$C_b(x,y)\ge1/(8\|x-y\|_\infty)$ for all
$b\ge\beta_{\mathrm{BKT}}$, including the endpoint.
Equation~\eqref{eq:infinite-comparison} with $q=1$ transfers this bound
to the diluted correlation and proves Corollary~\ref{cor:weak}.

\begin{remark}
The critical curve here is defined by the ordinary disorder average of
normalized two-point correlations. The proof does not identify it with
an exponential-decay threshold for a diluted model. Its phase-existence
input is exactly the averaged conclusion of Dario--Garban.
The finite-graph comparisons are independent of dimension; the site
proof uses bipartiteness and a uniform bound on the degree.
\end{remark}

\subsection{The susceptibility threshold}

The same order argument applies to divergence of the correlation sum.
Define, separately from \eqref{eq:critical},
\begin{equation}\label{eq:chi-critical}
 \beta_{\chi,\star}(p)
 =\inf\left\{\beta\ge0:
   \sum_{x\in\Z^2}F_\star(p,\beta;0,x)=\infty\right\}.
\end{equation}
Changing the site diagonal from $1$ to $p$ changes one finite summand
and leaves this definition unchanged.

\begin{corollary}\label{cor:chi-critical}
Theorem~\ref{thm:critical} holds with $\beta_{c,\star}$ replaced
throughout by $\beta_{\chi,\star}$. Moreover,
$\beta_{\chi,\star}(1)=\beta_{\mathrm{BKT}}$.
\end{corollary}

\begin{proof}
Clean exponential decay and Ginibre domination give
$\beta_{\chi,\star}(p)\ge\beta_{\mathrm{BKT}}>0$.
For every $p>p_{c,\star}$, the lower bound in
\cite[Theorem~1.3 and Remark~2, including the bond extension]{DG25}
is at least $c(p)/|x|$ at sufficiently large $\beta$, so the threshold
is finite. Divergence of a sum of nonnegative correlations is upward
closed in $\beta$ and $p$ and is preserved by
\eqref{eq:infinite-comparison}. The proof of
Theorem~\ref{thm:critical} therefore applies to this phase property as
well. At $p=1$, the clean lower bound at $\beta_{\mathrm{BKT}}$
gives equality.
\end{proof}

For either $c=\beta_{c,\star}$ or $c=\beta_{\chi,\star}$,
differentiating the two locally Lipschitz monotone functions gives
\begin{equation}\label{eq:slope}
 -\frac{\alpha(e^{2\kappa c(p)}-1)}{2\kappa p}
 \le c'(p)\le0
 \qquad\text{for almost every }p\in(p_{c,\star},1).
\end{equation}
The two thresholds are kept distinct. An unspecified polynomial lower
bound can have a summable exponent, while divergence of the sum need
not give a uniform pointwise polynomial lower bound. Their equality
requires an additional argument.

\section{Magnetic susceptibility and weak bond noise}\label{sec:physical}

The comparison has direct consequences for magnetic fluctuations in
finite boxes. A separate application to the Laplace transform controls
the possible shift of the algebraic-phase temperature under centered
bond noise. For the rest of the note set $b_0=\beta_{\mathrm{BKT}}$.

\subsection{Finite-size susceptibility}

For bond dilution let $z_x=e^{i\theta_x}$; for site dilution use the
physical occupied-site spin $z_x=\eta_x e^{i\theta_x}$. Put
$d_\bnd=1$ and $d_\sit=p$. Off-diagonal physical correlations agree
with $F_{G,\star}$ because a vacant endpoint has an independent Haar
angle. On the diagonal, $\E\av{|z_x|^2}=d_\star$.
For a finite volume $\Lambda$, define
\begin{equation}\label{eq:structure}
 S_\Lambda(0)=\frac1{|\Lambda|}\E\av{
       \left|\sum_{x\in\Lambda}z_x\right|^2},
 \qquad
 M_\Lambda=\frac1{|\Lambda|}\sum_{x\in\Lambda}z_x.
\end{equation}
The normalization is per original lattice site.

The clean input needed here is finite-volume. The proof in
\cite[Section~6]{vEL23} gives
\begin{equation}\label{eq:boundary-criterion}
 \sum_{y\in\partial G}C_{G,b}(x,y)\ge1
 \qquad\text{for }b\ge b_0
\end{equation}
for every finite domain $G$ containing $x$, with inner vertex boundary
$\partial G$. Here $C_{G,b}$ is the clean two-point function with
uniform coupling $b$ and free boundary conditions. A boundary sum below
one implies exponential decay by the Lieb--Rivasseau inequality;
continuity of each finite-domain sum excludes it at $b=b_0$ as well.
See the critical definition and its consequences in Equation~(16) of
the cited arXiv version, Equation~(17) in the published version.

\begin{proposition}[Finite-size magnetic fluctuations]\label{prop:susceptibility}
Assume
\begin{equation}\label{eq:protected}
 p^\alpha H_\kappa(\beta)\ge H_\kappa(b_0),
\end{equation}
with the bond or site parameters in Theorem~\ref{thm:critical}.
For the free square $\Lambda_R=[-R,R]^2\cap\Z^2$, with integer $R\ge1$,
\begin{equation}\label{eq:local-chi}
 \sum_{y\in\Lambda_R}\E\av{z_0\overline z_y}_{\Lambda_R}
 \ge R+d_\star.
\end{equation}
Writing $m=\lfloor R/2\rfloor$, $N=(2R+1)^2$ and
$N_{\mathrm{core}}=(2m+1)^2$, we also have
\begin{equation}\label{eq:free-chi}
 S_{\Lambda_R}(0)\ge d_\star+\frac{N_{\mathrm{core}}}{N}\,m.
\end{equation}
On an $L\times L$ square torus with integer $L\ge3$, put
$R=\lfloor(L-1)/2\rfloor$. Then
\begin{equation}\label{eq:torus-chi}
 S_L(0)\ge R+d_\star,\qquad
 \E\av{|M_L|^2}\ge\frac{R+d_\star}{L^2},\qquad
 \chi_{\mathrm{mag},L}\ge\frac\beta2(R+d_\star).
\end{equation}
Here $\chi_{\mathrm{mag},L}$ is the zero-field response of one
Cartesian magnetization component per original site to the energy
term $-h\sum_x\operatorname{Re}z_x$.
\end{proposition}

\begin{proof}
Put $b=B_\kappa(p^\alpha,\beta)\ge b_0$.
For each $1\le r\le R$, apply \eqref{eq:boundary-criterion} to
$\Lambda_r$ at coupling $b$, and then use
Proposition~\ref{prop:density} with $q=1$ to obtain
\[
 \sum_{y\in\partial\Lambda_r}
 F_{\Lambda_r,\star}(p,\beta;0,y)\ge1.
\]
Ginibre monotonicity preserves each bound when $\Lambda_r$ is enlarged
to $\Lambda_R$. The shells are disjoint, so their sum and the physical
diagonal give \eqref{eq:local-chi}; see Figure~\ref{fig:shells}.
For every $x\in\Lambda_m$, the square $x+\Lambda_m$ lies in
$\Lambda_R$. Sum the same bound over these centers, retain the diagonal
at all other centers, and use positivity of the remaining correlations.
This proves \eqref{eq:free-chi}; when $m=0$, it is just the diagonal bound.

On the torus, a free square of radius $R$ embeds around every vertex.
It is bipartite even if $L$ is odd. Apply \eqref{eq:local-chi} on this
square before adding the remaining torus edges, which increase the
correlations by Ginibre. Sum over centers to obtain the first bound in
\eqref{eq:torus-chi}. The second follows from \eqref{eq:structure}.
Rotation symmetry gives the exact identity
$\chi_{\mathrm{mag},L}=\beta S_L(0)/2$, proving the last bound.
\end{proof}

\begin{figure}[tbp]
\centering
\resizebox{.98\linewidth}{!}{%
\begin{tikzpicture}[font=\fontsize{9}{11}\selectfont,text=XYInk,>=Stealth]
  \node[anchor=west,font=\fontsize{10}{12}\selectfont\bfseries] at (-.1,5.23) {One unit from each shell};
  \begin{scope}[shift={(2.05,2.65)},x=.52cm,y=.52cm]
    \draw[step=1,XYMuted!65,line width=.3pt] (-3,-3) grid (3,3);
    \draw[XYGreen,line width=1.5pt] (-3,-3) rectangle (3,3);
    \draw[XYBlue,line width=1.5pt] (-2,-2) rectangle (2,2);
    \draw[XYRust,line width=1.5pt] (-1,-1) rectangle (1,1);
    \foreach \x in {-3,...,3} {\foreach \y in {-3,...,3} {
      \fill[XYInk!45] (\x,\y) circle (1pt);}}
    \fill[XYInk] (0,0) circle (2.8pt);
    \node[below right=3pt,font=\fontsize{8}{9.5}\selectfont,fill=white,inner sep=1pt] at (0,0) {$0$};
    \node[fill=white,inner sep=2pt,text=XYRust,font=\fontsize{7}{8}\selectfont] at (0,-1) {$\partial\Lambda_1$};
    \node[fill=white,inner sep=2pt,text=XYBlue,font=\fontsize{7}{8}\selectfont] at (0,-2) {$\partial\Lambda_2$};
    \node[fill=white,inner sep=2pt,text=XYGreen,font=\fontsize{7}{8}\selectfont] at (0,-3) {$\partial\Lambda_3$};
  \end{scope}
  \node[align=center,font=\fontsize{8}{9.5}\selectfont,text width=4.2cm] at (2.05,.37)
    {Disjoint inner boundaries\\inside a free square $\Lambda_R$};
  \draw[XYMuted] (4.6,-.05)--(4.6,5.46);
  \node[anchor=west,font=\fontsize{10}{12}\selectfont\bfseries] at (5,5.23) {The finite-volume argument};
  \node[anchor=west,font=\fontsize{8}{9.5}\selectfont] at (5,4.63)
    {Protected region: $p^\alpha H_\kappa(\beta)\ge H_\kappa(b_0)$};
  \node[anchor=west,align=left,text width=7.15cm] at (5,3.64)
    {\textbf{1. Each shell in its own box}\\[4pt]
     $\displaystyle\sum_{y\in\partial\Lambda_r}
       F_{\Lambda_r,\star}(p,\beta;0,y)\ge1$\\[4pt]
     \fontsize{8}{9.5}\selectfont Clean finite-domain criterion $+$ disorder comparison};
  \draw[->,XYBlue,line width=1pt] (5.18,2.84)--(5.18,2.29);
  \node[anchor=west,font=\fontsize{8}{9.5}\selectfont,text=XYBlue] at (5.45,2.56)
    {enlarge $\Lambda_r$ to $\Lambda_R$ by Ginibre};
  \node[anchor=west,align=left,text width=7.15cm] at (5,1.48)
    {\textbf{2. Sum shells, then add the diagonal}\\[4pt]
     $\displaystyle\sum_{y\in\Lambda_R}
       \mathbb E\langle z_0\bar z_y\rangle_{\Lambda_R}\ge R+d_*$\\[4pt]
     \fontsize{8}{9.5}\selectfont $d_*=1$ for bonds; $d_*=p$ for occupied-site spins};
  \node[anchor=west,font=\fontsize{8}{9.5}\selectfont,text width=11.9cm,align=left] at (-.1,-.78)
    {Summing the same local bound over centers yields a bulk susceptibility bound.\\
     The shell estimate is a sum over vertices; it is not a pointwise finite-box lower bound.};
\end{tikzpicture}
}
\caption{The finite-domain boundary criterion and the disorder comparison give
at least one unit from each shell in its own free square. Ginibre monotonicity
preserves these bounds in the outer square. Summing disjoint shells and adding
the physical diagonal gives \eqref{eq:local-chi}; summing over bulk centers
gives \eqref{eq:free-chi}. The shell estimate does not assert a pointwise
finite-box lower bound.}
\label{fig:shells}
\end{figure}

The finite-domain criterion is essential: the infinite-volume
$1/|x|$ lower bound does not by itself give a lower bound in a free
finite box. Proposition~\ref{prop:susceptibility} implies at least
linear growth in box side length of the averaged zero-momentum
structure factor and susceptibility in the region \eqref{eq:protected}.
It does not determine their critical exponent. For response to a
dimensionless field, the factor $\beta$ is omitted.

There is also an almost-sure bulk consequence. Let $S_\Lambda^\omega(0)$
denote \eqref{eq:structure} in a fixed iid disorder sample, without
$\E$. Fix $r$ and average the local susceptibility in translated
free squares of radius $r$ over centers in a growing free box.
Ginibre bounds $S_\Lambda^\omega(0)$ below by this spatial average,
up to the boundary fraction. The averaged local quantity is bounded,
so the ergodic theorem and \eqref{eq:local-chi} give
$\liminf_{\Lambda\uparrow\Z^2}S_\Lambda^\omega(0)\ge r+d_\star$
almost surely along squares. Intersecting the probability-one events
for integer $r$ and letting $r\to\infty$ gives bulk divergence.
This is the spatial-ergodic mechanism of
\cite[Corollary~1.4 and Section~5.3]{DG25}, applied in the explicit
region \eqref{eq:protected}. It gives neither a samplewise linear
rate nor divergence of the fixed-origin quenched susceptibility.

\subsection{Quadratic protection against centered bond noise}

Let the iid bond strengths be $J_e=J_0+\varepsilon X_e$, where
$J_0>0$, $X$ is bounded, $\E X=0$ and $v=\operatorname{Var}(X)$.
Take $|\varepsilon|$ small enough that $J_e\ge J_0/2$.
The couplings in \eqref{eq:gibbs} are $\beta J_e$.
Define $\beta_{\mathrm{alg}}(\varepsilon)$ by the polynomial-lower-bound
criterion \eqref{eq:critical} for this coupling law, and put
$T_{\mathrm{alg}}(\varepsilon)=1/\beta_{\mathrm{alg}}(\varepsilon)$
and $T_0=J_0/b_0$.

\begin{corollary}\label{cor:noise}
There is a finite constant $C$, depending on $J_0$ and the law of $X$,
such that, for all sufficiently small $|\varepsilon|$,
\begin{equation}\label{eq:noise}
 T_{\mathrm{alg}}(\varepsilon)
 \ge T_0-\frac{v}{J_0}\varepsilon^2-C|\varepsilon|^3.
\end{equation}
\end{corollary}

\begin{proof}
Theorem~\ref{thm:random} gives the clean comparison coupling
\begin{equation}\label{eq:noise-laplace}
 \Phi(\beta,\varepsilon)
 =\beta J_0-\frac12\log\E e^{-2\beta\varepsilon X}
 =\beta J_0-\beta^2v\varepsilon^2+O(|\varepsilon|^3).
\end{equation}
Boundedness of $X$ makes the remainder uniform for $\beta$ in a compact
neighborhood of $b_0/J_0$. Since
$\partial_\beta\Phi(b_0/J_0,0)=J_0>0$, local inversion gives a
solution of $\Phi(\beta_*(\varepsilon),\varepsilon)=b_0$ satisfying
\[
 \beta_*(\varepsilon)
 =\frac{b_0}{J_0}+\frac{b_0^2v}{J_0^3}\varepsilon^2
   +O(|\varepsilon|^3).
\]
The clean endpoint lower bound implies
$\beta_{\mathrm{alg}}(\varepsilon)\le\beta_*(\varepsilon)$.
The threshold is positive by comparison with a finite clean upper
bound on the couplings. Taking reciprocals and expanding proves
\eqref{eq:noise}.
\end{proof}

Thus centered bounded noise can suppress the algebraic-phase temperature
by at most a quadratic amount in its amplitude. This is a one-sided
bound: it neither proves that suppression occurs nor determines its
actual quadratic coefficient. The effective coupling in
\eqref{eq:noise-laplace} is a comparison parameter, not an identification
of the helicity modulus.

\section{Perturbing the coupling distribution}\label{sec:law}

The thinning argument also controls changes to an entire bounded
coupling law. Let $\mu$ be supported on $[0,M]$, $M>0$, and write
$F_\mu(\beta;x,y)$ for the averaged free infinite-volume correlation
with iid couplings $\beta J_e$, $J_e\sim\mu$.
Let $c(\mu)$ denote either the algebraic threshold defined as in
\eqref{eq:critical} or the susceptibility threshold defined as in
\eqref{eq:chi-critical}, keeping the choice fixed.
For two laws set
$\delta=\|\nu-\mu\|_{\mathrm{TV}}=\sup_A|\nu(A)-\mu(A)|$.

\begin{theorem}[Regularity in the coupling law]\label{thm:law}
Suppose $0<c(\mu)<\infty$ and all competing laws $\nu$ are supported
on $[0,M]$.
If $a=\mu(\{M\})>0$, then, for
$\delta<a e^{-2Mc(\mu)}$,
\begin{equation}\label{eq:law-bounds}
 \frac{H_M(c(\mu))}{1+\delta/a}
 \le H_M(c(\nu))
 \le\frac{H_M(c(\mu))}{1-\delta/a}<1.
\end{equation}
In particular $c$ is locally Lipschitz in total variation at $\mu$,
and, with $\mu$ fixed,
\begin{equation}\label{eq:law-expansion}
 |c(\nu)-c(\mu)|
 \le\frac{e^{2Mc(\mu)}-1}{2Ma}\,\delta+O(\delta^2).
\end{equation}
More generally, if $\operatorname{ess\,sup}\mu=M$, then $c$ is
continuous at $\mu$ in total variation among laws on $[0,M]$,
without requiring an atom at $M$.
\end{theorem}

For example, the Lipschitz conclusion holds around Bernoulli($p$)
for every $p>p_{c,\bnd}$, and allows arbitrary nearby laws on $[0,1]$.
An atom at zero is permitted. The common upper cap in the theorem is
part of the hypothesis.

\subsection{Thinning bounded strengths}

Define
\begin{equation}\label{eq:psi-U}
 \psi_M(r,t)=\frac1M B_1(r,Mt)
 =-\frac1{2M}\log(1-rH_M(t)),\qquad
 U_M(r,t)=-\frac1{2M}\log(1-H_M(t)/r),
\end{equation}
where $U_M(r,t)$ is used only when $r>H_M(t)$.
Conditioning on $J$ and thinning by independent Bernoulli($r$)
variables $\eta_e$, Theorem~\ref{thm:random} replaces
$\beta\eta_eJ_e$ by $B_1(r,\beta J_e)$.
For $0<r<1$,
\[
 \partial_t^2B_1(r,t)
 =-\frac{2r(1-r)e^{-2t}}{(1-r+re^{-2t})^2}\le0.
\]
Since $B_1(r,0)=0$, concavity gives
$B_1(r,\beta J_e)\ge(J_e/M)B_1(r,\beta M)$.
Ginibre and the free-volume limit therefore imply
\begin{equation}\label{eq:law-thinning}
 F_{\operatorname{law}(\eta J)}(\beta;x,y)
 \ge F_\mu(\psi_M(r,\beta);x,y).
\end{equation}
The cases $r=0,1$ follow directly.

If $\nu$ stochastically dominates $\operatorname{law}(\eta J)$,
product coupling and \eqref{eq:law-thinning} give
\begin{equation}\label{eq:law-threshold}
 c(\nu)\le U_M(r,c(\mu))
 \quad\text{when }r>H_M(c(\mu)).
\end{equation}
Indeed, first take $t>c(\mu)$ with $H_M(t)<r$ and apply the
comparison at $U_M(r,t)$; then let $t\downarrow c(\mu)$.
This avoids assuming critical attainment.

\subsection{Proof of Theorem~\ref{thm:law}}

Suppose first that $a=\mu(\{M\})>0$.
Set $r=1-\delta/a$ and $s=(1+\delta/a)^{-1}$.
For $0<t\le M$, the tail $T_\mu(t)=\mu([t,M])$ is at least $a$,
so total variation gives
\[
 T_\nu(t)\ge T_\mu(t)-\delta\ge rT_\mu(t),\qquad
 T_\nu(t)\le T_\mu(t)+\delta\le T_\mu(t)/s.
\]
Thus $\nu$ dominates Bernoulli($r$) thinning of $\mu$, while
$\mu$ dominates Bernoulli($s$) thinning of $\nu$.
The hypothesis $\delta<a e^{-2Mc(\mu)}$ implies
$r>H_M(c(\mu))$, so \eqref{eq:law-threshold} makes $c(\nu)$ finite
and proves the upper bound in \eqref{eq:law-bounds}.
For the lower bound, if $H_M(c(\nu))\ge s$, then
$H_M(c(\nu))\ge sH_M(c(\mu))$ immediately.
Otherwise apply \eqref{eq:law-threshold} in the reverse direction
to obtain $H_M(c(\mu))\le H_M(c(\nu))/s$.
Inverting the bounds and expanding gives \eqref{eq:law-expansion}.
A sufficiently small TV neighborhood retains a top atom bounded below
and thresholds bounded above. Applying these estimates with any law
in that neighborhood as reference gives a uniform pairwise Lipschitz
constant there.

Now assume only $\operatorname{ess\,sup}\mu=M$.
Fix $\lambda>1$ and let $a_\lambda=\mu([M/\lambda,M])>0$.
For $\delta<a_\lambda$, put
$r=1-\delta/a_\lambda$ and $s=(1+\delta/a_\lambda)^{-1}$.
Tail comparison yields
\begin{equation}\label{eq:scaled-tails}
 \nu\succeq\operatorname{law}(\operatorname{Bern}(r)J_\mu/\lambda),
 \qquad
 \operatorname{law}(\lambda J_\mu)
 \succeq\operatorname{law}(\operatorname{Bern}(s)J_\nu).
\end{equation}
For the first inequality, when $t\le M/\lambda$ use
$T_\nu(t)\ge T_\mu(t)-\delta\ge rT_\mu(t)\ge rT_\mu(\lambda t)$;
above $M/\lambda$ the target tail vanishes.
For the second, when $t\le M$ use
$T_\nu(t)\le T_\mu(t/\lambda)+\delta\le T_\mu(t/\lambda)/s$;
above $M$ the target tail vanishes.

The scaling identity
$c(\operatorname{law}(\lambda J_\mu))=c(\mu)/\lambda$
and \eqref{eq:law-threshold}, with the respective support caps,
give, for $\delta$ small enough that $r>H_M(c(\mu))$,
\begin{equation}\label{eq:law-continuity}
 \psi_M(s,c(\mu)/\lambda)
 \le c(\nu)\le\lambda U_M(r,c(\mu)).
\end{equation}
The upper bound uses cap $M/\lambda$ for $J_\mu/\lambda$.
For the lower bound use cap $M$ for $J_\nu$, separating the trivial
case $H_M(c(\nu))\ge s$ as above.
Letting $\delta\to0$ at fixed $\lambda$ gives bounds
$c(\mu)/\lambda$ and $\lambda c(\mu)$; then let $\lambda\downarrow1$.
This proves continuity.

The finiteness hypothesis holds, for example, whenever
$\mu((0,M])>p_{c,\bnd}$. Indeed, choose $a>0$ with
$\mu([a,M])>p_{c,\bnd}$ and compare with strength-$a$ Bernoulli
bonds, using \cite[Theorem~1.3 and Remark~2]{DG25}.
Positivity follows from clean domination at strength $M$.

\subsection{Disorder with a fixed dependence range}

Independence can also be weakened when the goal is stability near the
clean model. Call an occupation field $\ell$-dependent if families
indexed by sets at distance greater than $\ell$ are independent.
For bond indexing one may use the lattice of doubled edge midpoints.

\begin{corollary}\label{cor:finite-range}
Fix $\ell<\infty$ and $\beta>b_0$.
There is $\widehat p(\ell,\beta)<1$ such that every
$\ell$-dependent bond-occupation field on $\Z^2$ whose marginals
are all at least $\widehat p$ satisfies
\[
 \E\av{\cos(\theta_x-\theta_y)}_{\beta,\omega}
 \ge\frac1{8\|x-y\|_\infty}\qquad(x\ne y).
\]
The analogous statement holds for site occupation.
Neither stationarity nor positive association is required.
\end{corollary}

\begin{proof}
Choose $\rho<1$ with
$\rho^\alpha H_\kappa(\beta)>H_\kappa(b_0)$.
The Liggett--Schonmann--Stacey theorem \cite{LSS97}, in the formulation
of \cite[Theorem~2.9]{DG25}, gives a marginal threshold below one
that ensures stochastic domination of iid Bernoulli($\rho$).
For bonds, index horizontal and vertical edges by
$(2x+1,2y)$ and $(2x,2y+1)$, respectively. They form a subset of
$\Z^2$, covered by that theorem, with a fixed dependence range.
Ginibre and Corollary~\ref{cor:weak} transfer the iid lower bound.
\end{proof}

For nonnegative strengths, the same argument applies through
$J_e\ge a\,1_{\{J_e\ge a\}}$ when $\beta a>b_0$ and
$\inf_e\mathbb P(J_e\ge a)$ is sufficiently close to one.
The dependence range must remain fixed. High marginals alone are
insufficient: a uniformly shifted periodic grid of deleted separating
bonds can leave only finite square components while each bond's
retention probability tends to one.

\section{Further questions}\label{sec:questions}

The support restriction in Theorem~\ref{thm:law} leaves a concrete
extension problem. If $\nu_n\to\mu$ in total variation, all supported
on a fixed $[0,M]$ but $M>\operatorname{ess\,sup}\mu$, must
$c(\nu_n)\to c(\mu)$ at a finite positive threshold?
The tail proof above fails where the reference law has no upper-tail
mass. Continuity under weak convergence, equivalently $W_1$ convergence
on this common compact support, is a further question.

A local quantitative problem concerns the star estimate. For a full
star with interaction $tS$, no other interaction at its center, and
arbitrary nonnegative exterior couplings, Lemma~\ref{lem:convex}
gives
\[
 f(\beta)-f(0)
 \le\frac{e^{2\Delta\beta}-1}{2\Delta}\,f'(\beta).
\]
Determining the optimal uniform coefficient, or improving its
large-$\beta$ order, would sharpen the site-dilution comparison.
The inequality should include observables with $f'(\beta)=0$.

For supercritical Bernoulli bond dilution in \eqref{eq:protected},
a different question is whether
\[
 \sum_x C_\omega(\beta;0,x)=\infty
 \quad\text{almost surely conditional on }0\in\mathcal C_\infty,
\]
where $C_\omega$ is the fixed-disorder free infinite-volume correlation
and $\mathcal C_\infty$ is the unique infinite open cluster.
The almost-sure bulk statement in Section~\ref{sec:physical} does not
imply this fixed-origin conclusion. A uniform quenched lower bound
$c_\omega/(1+\|x\|_\infty)$ on infinite-cluster vertices, with
$c_\omega>0$, would suffice by the cluster's positive asymptotic density.

Finally, define
\[
 \beta_{\exp,\star}(p)=\sup\left\{\beta\ge0:
 \begin{array}{l}
 \text{there exist }A,a>0\text{ such that}\\[-2pt]
 F_\star(p,\beta;0,x)\le Ae^{-a\|x\|_\infty}
 \quad\text{for every }x\ne0
 \end{array}\right\}.
\]
Equality of $\beta_{\exp,\star}$, $\beta_{\chi,\star}$ and
$\beta_{c,\star}$ would establish a diluted sharpness statement.
The present comparisons do not supply a diluted counterpart of the
clean finite-size criterion that links these three definitions.

\end{document}